\documentclass[english,12pt,reqno]{smfart}
\usepackage{color} 
\usepackage{amsthm} 
\usepackage{amsmath}
\usepackage{amsfonts} 
\usepackage{amstext} 
\usepackage[latin1]{inputenc}
\usepackage{amscd}
\usepackage{latexsym}
\usepackage{bm}
\usepackage{amssymb}
\usepackage[all]{xy}
\usepackage{euscript}
\usepackage{a4wide}
\usepackage{amsmath,amssymb,graphicx}
\usepackage{amssymb}
\usepackage{amsmath}

\newtheorem{thm}{Theorem}
\newtheorem{defi}{Definition}
\newtheorem{lem}{Lemma}

\newtheorem{rema}{Remark}

\def\cal{\mathcal}

\def\di{\displaystyle}
 \newcommand{\N}{\mathbb{N}}
 \newcommand{\R}{\mathbb{R}}
\def \rr        {{\mathbb R}}

\renewcommand \bar[1]   {\overline{#1}}

\def \oo    {\Omega}           
\def \bo    {\partial\Omega}   

\def \L    {L}             

\def\cal{\mathcal}

\def\di{\displaystyle}
\def\div{\mbox{\rm div}}

\newcommand{\dpar}[2]{\frac{\partial #1}{\partial #2}}

\newcommand{\mc}[1]{\mathcal{#1}}
\newcommand{\Dp}[0]{\, \mathcal{D}^{\alpha}_+ \,}
\newcommand{\Dm}[0]{\, \mathcal{D}^{\alpha}_- \,}

\newcommand{\Dcp}[0]{\, {}^c \mathcal{D}^{\alpha}_+ \,}
\newcommand{\Dcm}[0]{\, {}^c \mathcal{D}^{\alpha}_- \,}
\newcommand{\Da}[0]{\, {}^c \mathcal{D}^{\alpha} }
\newcommand{\Dpv}[1]{\, \mathcal{D}^{#1}_+ \,}
\newcommand{\Dmv}[1]{\, \mathcal{D}^{#1}_- \,}

\newcommand{\Dcpv}[1]{\, {}^c \mathcal{D}^{#1}_+ \,}

\newcommand{\Ipv}[1]{\, \mathcal{I}^{#1}_+ \,}
\newcommand{\Imv}[1]{\, \mathcal{I}^{#1}_- \,}
\newcommand{\divg}{\text{\upshape div}}

\begin{document}

\title{Lagrangian for the convection-diffusion equation}
\author{Jacky Cresson}
\author{Isabelle Greff}
\author{Pierre Inizan}

\maketitle

\pagestyle{myheadings}
\markboth{Lagrangian for the convection-diffusion equation}
{J. Cresson, I. Greff, P. Inizan}

\begin{abstract}
Using the asymmetric fractional calculus of variations, 
we derive a fractional 
Lagrangian variational formulation of the convection-diffusion 
equation in the special case of constant coefficients.  
\end{abstract}


\begin{tiny}
\tableofcontents
\end{tiny}

\part{Introduction}
\label{part1}
\section{Introduction}
The convection-diffusion equation occurs in many physical problems such as porous media, engineering, geophysics. It could model the dispersion of a pollutant in a river estuary, or groundwater transport, atmospheric pollution, concentration of electron inducing an electric current, heat transfer in a heated body. 
As many PDEs, the solution exists under conditions, but is often not known explicitly. Even for linear convection-diffusion equation numerical schemes are not always well understood. 
It is still a challenging problem 
to obtain efficient and robust numerical schemes
to solve the 
convection-diffusion equation due in particular to the mixing 
between two different types
of behavior, namely the convective and diffusive regimes.
Let us first recall the equation. Let $\oo\subset \rr^d$ be an open subset with a Lipschitz-continuous boundary $\bo$. The time domain $[a,b]$, $0\leq a\leq b$ is arbitrary, but fixed. Let us consider the general linear parabolic equation of second order: 
\begin{equation}
\label{CDunst}
\begin{array}{lll}
u_t + (\gamma \cdot \nabla) u -\div(K \cdot \nabla u)+ \beta u = f(t,x)\,
\quad \text{ in } (a,b]\times\Omega   ,\\
u(t,x) = 0  \quad \text{ in } (a,b] \times\partial\Omega   ,\\
u(a,x) = u_0(x) \quad \text{ in } \Omega\,,
\end{array}
\end{equation}
where  $\gamma \in \R^d,K \in \R^{d\times d}, \beta\in \R$.

As an example $u$ is the concentration of a pollutant, transported by a flow of velocity $\gamma \in \R^d$. The tensor $K$ represents the
diffusivity of the pollutant specie. The creation or destruction of the specie can be taken into account via $\beta u$, and $f$ is the source term.
The unknown $u$ is both depending on time and space. \\ 
We assume the coefficients are smooth, bounded and satisfying the following properties:
\begin{itemize}
\item $K$ is symmetric, uniformly positive definite s.t. 
\begin{equation*}
K\in \mathcal{C}^0(a,b;L^{\infty}(\oo)^{d\times d})\quad 
{\text and }\quad 
\exists \lambda_1, \lambda_2 >0: \lambda_1 \vert \xi\vert^2 < \xi K\xi^T
< \lambda_2 \vert \xi\vert^2 
\end{equation*}
\item
the convection $\gamma$ is such that
\begin{equation*}
\gamma\in \mathcal{C}^0(a,b;W^{1,\infty}(\oo)^{d}) \quad {\text and } 
\quad \div \gamma = 0 
\end{equation*}

\item the reaction $\beta$ is non-negative 
\begin{equation*}
\beta\in \mathcal{C}^0(a,b;L^{\infty}(\oo)) \quad {\text and } 
\quad \exists \beta_0 \quad
\beta\geq \beta_0\,.
\end{equation*}
\end{itemize}
These assumptions guarantee that the problem \eqref{CDunst} is well 
posed for $f\in L^2(a,b;H^{-1}(\Omega))$, and every $u_0 \in L^2(\Omega)$.\\
Let us notice that in the special case of absence of convection, 
when $\gamma=0$, the stationary convection-diffusion equation is 
simply the Poisson equation. It is well known that the Poisson 
equation derives from a variational principle also called 
least-action principle. This means that the solution $u$ of 
the Poisson equation is a minimizer of the following Lagrangian functional
\begin{eqnarray*}
\label{L}
\cal{L} & : &  H_0^1(\oo) \longrightarrow \R \\
& & v\mapsto 
\cal{L}(v) = \int_{\oo} \frac{1}{2} (K \cdot \nabla  v) \cdot \nabla  v \,dx 
- \int_{\oo} fv\,dx\,.
\end{eqnarray*}
This is not the case of the convection-diffusion equation. 
For instance, the stationary convection-diffusion equation 
admits the following weak formulation:
\begin{equation}
\label{eq:weakCD}
\int_{\oo}  (K \cdot \nabla  v) \cdot \nabla   \phi \,dx + \int_{\oo} (\gamma \cdot \nabla) u \,\phi + \beta u \phi \,dx = \int_{\oo} f\phi\,dx\,,\quad \text{ for any } \phi \in 
H_0^1(\oo).
\end{equation}
Nevertheless the advective term is not symmetric in $u$ and $\phi$. As a consequence, the weak formulation \eqref{eq:weakCD} does not derive from a potential, \cite{Vainberg}. This can also be seen as the convection-diffusion equation does not satisfy the so-called Helmholtz conditions, i.e. that the Fr\'echet derivative of the Euler-Lagrange expression is not  self-adjoint. We refer to (\cite{Olver}, Thm. 5.92, \-p.364) for more details.

Let us note that there were some attempt to construct variational formulation for the convection-diffusion equation by Ortiz \cite{Ortiz}, where he resorts to a local transformation of the solution by use of a ``dual'' problem.
In this paper, we prove that the solutions of the convection-diffusion equation correspond to critical points of a fractional Lagrangian functional. The idea to use fractional derivatives in order to bypass classical obstruction to the existence of a Lagrangian functional was discussed by Riewe (\cite{Riewe1},\cite{Riewe2}). However, for technical reasons his idea cannot provide a fractional variational formulation of the convection-diffusion equation. This difficulty was solved recently by Cresson and Inizan in \cite{ci}. They introduce the asymmetric fractional calculus of variations, and obtain an explicit fractional Lagrangian.\\

The quest for such a variational formulation is not only an abstract mathematical problem in calculus of variations. Due to the existence of different flow regime which depends on a coefficient called the Reynolds number, many problems arise in the construction of numerical schemes for 
the convection-diffusion equation. 
The main difficulty is to avoid non physical spurious oscillations 
which appear in numerical experiment for high value of 
the Reynolds number. Our idea is to use the previous 
Lagrangian variational formulation in order to build 
a fractional version of a variational integrators.
In the classical case (see \cite{Lubich})
the variational integrator possess good numerical properties.
A first attempt to construct fractional variational integrators
is done in \cite{cbgi}.\\

The outline of the paper is as follow. First we give some notations. 
In Part \ref{part2}, we explain Riewe's approach in the case of dissipative 
equations,
the obstructions to the 
existence of a Lagrangian for the convection-diffusion 
equation and how it is related to the irreversibility of the equation. 
Part \ref{part3} is devoted to recall basics about fractional 
calculus  and introduce the asymmetric fractional principle of
variations. 
Finally in Part \ref{part4}, we give a Lagrangian 
associated to the convection-diffusion equation based on the
fractional derivatives.
\section{Notations and assumptions}
\label{sect:notations}
\subsection{Domain}
Let $d\in \N$. We consider a smooth $d$-dimensional bounded convex domain $\Omega$ with boundary $\partial \Omega$. Let  $(e_1,\ldots,e_d)$ be the canonical basis for $\R^d$. For any $x \in \R^d$, we denote by $x_i$ the $i$-th component of $x$ in the canonical basis of $\R^d$.
Let $1 \leq i \leq d$ and $x \in \overline{\Omega}$. We denote by $\delta_{i,x}$ the straight line of $\R^d$ defined by  
\begin{equation*}
\delta_{i,x} = x + \mbox{\rm Span}(e_i)
\end{equation*} 
and $\Omega_{i,x} = \overline{\Omega} \cap \delta_{i,x}$. 
As $\Omega$ is bounded and convex, $\Omega_{i,x}$ is a segment. 
Then, it exists $a_{i,x} \leq b_{i,x}$ such that 
\begin{equation}
\label{omegaix}
\Omega_{i,x}: =\{ (x_1 ,\dots ,x_{i-1} ,t,x_{i+1} ,\dots ,x_d ) \mid t\in [a_{i,x} ,b_{i,x} ] \} .
\end{equation}
\subsection{Functional sets}
\label{subsection:functional}
For two sets $A$ and $B$, $\mc{F}(A,B)$ denotes the vector space of 
functions $f \, : \, A \rightarrow B$. Let $a,b \in \R$, $a<b$. 
Let $m, n \in \N^*$ and $p, q \in \N$. 
Let $\mc{U}$ be an open subset of $\R^m$ or the finite 
interval $[a,b]$. The vector space of functions 
$\mc{U} \rightarrow \R^n$ of class $C^p$ is denoted 
by $C^p(\mc{U})$. 
For $f\in \mc{F}([a,b]\times\Omega,\R)$,
we denote by:
\begin{eqnarray*}
\begin{array}{ccc}
\forall t\in [a,b]\quad f_t  :& \Omega & \longrightarrow \R\\
& x& \mapsto f(t,x)
\end{array} 
&\quad \text{and}\quad
\begin{array}{ccc}
\forall x\in \Omega\quad
f_x  :&  [a,b]&  \longrightarrow \R\,.\\
& t &\mapsto f(t,x)
\end{array} 
\end{eqnarray*}
Let $C^{p,q} ([a,b]\times\Omega )$ and $C^p ([a,b]\times\Omega )$
the functional spaces defined as follow:
\begin{equation*}
C^{p,q} ([a,b]\times\Omega ): =\{f \in \mc{F}([a,b]\times\Omega,\R)
\vert\,
\forall t \in [a,b], \, f_t \in C^p(\Omega),
\forall  x \in \Omega, \, f_x \in C^q ([a,b]) \},
\end{equation*}
and 
$C^p ([a,b]\times\Omega ):= C^{p,p} ([a,b]\times\Omega )$ when $q=p$.
Let $C^p_0(\Omega): = \{ f \in C^p(\Omega) \; | \; f = 0 \text{ on } \partial \Omega \}$.\\
For $p=0$, we introduce the following vector spaces:
\begin{equation*}
C^0_+([a,b]) := \{ f \in C^0([a,b]) \; | \; f(a)=0 \},
\end{equation*} 
\begin{equation*}
C^0_-([a,b]) := \{ f \in C^0([a,b]) \; | \; f(b)=0 \},
\end{equation*} 
and for $p \geq 1$, let
\begin{equation*}
C^p_+([a,b]) := \{ f \in C^p([a,b]) \; | \; f^{(k)}(a)=0, \; 0 \leq k \leq p-1 \},
\end{equation*} 
\begin{equation*}
C^p_-([a,b]) := \{ f \in C^p([a,b]) \; | \; f^{(k)}(b)=0, \; 0 \leq k \leq p-1 \}.
\end{equation*} 
\begin{equation*}
C^p_0([a,b]): = C^p_+([a,b]) \cap C^p_-([a,b]).
\end{equation*} 
The set of absolutely continuous functions over $[a,b]$ is 
denoted by $AC ([a,b])$ and $AC^{p+1} ([a,b])$ is the set defined by
$$AC^{p+1} ([a,b]) :=\{ f\in C^{p} ([a,b]), \, f^{(p)} \in AC([a,b]) \} .$$
Then $C^p ([a,b]) \subset AC^p ([a,b])$.
A natural functional space for the study of classical PDEs is
\begin{equation*}
F^{p,q} ([a,b]\times\bar\Omega ): =\{f \in \mc{F}([a,b]\times\bar\Omega,\R)
\vert\,
\forall\, t \in [a,b], \, f_t \in C^p(\bar\Omega),
 \forall \, x \in \bar\Omega,\, f_x \in AC^q ([a,b]) \}.
\end{equation*}
Let $F^p ([a,b] \times \bar{\Omega} ):= F^{p,p} ([a,b]\times\bar\Omega )$ 
when $q=p$.
For any integer $m\geq0$, $H^m(\oo)$ denotes the Sobolev space of order $m$.

\subsection{Fields}
In this paper we are interested in fields  $u$ depending on time $t\in [a,b]$ and space $x\in \Omega$:
\begin{equation*}
\begin{array}{cccl}
u \, : & [a,b]\times \overline{\Omega}   & \longrightarrow & \; \R\,.  \\
               &   (t,x)    & \longmapsto & u(t,x)
\end{array}
\end{equation*} 
The notation $\nabla u(t,x) \in \R^d$ is the gradient of 
$x \mapsto u(t,x)$ and $\partial_t u(t,x) \in \R$ the partial 
derivative of $u$ according to $t$ and $\partial_{x_i} f$ 
the derivative of $f$ in the $i$-th space-variable.
\noindent
The divergence of a vector field $F=(F_1 ,\dots ,F_d )$ is $\div F =\di\sum_{i=1}^d \partial_{x_i} F_i$. 
\noindent
Let $v \, : \, \Omega \rightarrow \R$, $1 \leq i \leq d$ and $x \in \overline{\Omega}$. We denote by $v_{i,x}$ the function defined by 
\begin{equation}
\begin{array}{cccl}
v_{i,x} \, : & \Omega_{i,x} & \longrightarrow & \; \R  \\
 &   y    & \longmapsto & v(x_1,\ldots, x_{i-1}, y, x_{i+1}, \ldots, x_d).
\end{array}
\end{equation}
\noindent
If $v \, : \, \overline{\Omega} \rightarrow \R$, we denote also 
by $v$ its extension to $\R^d$ such that $v(x)=0$ if $x \in \R^d \backslash \overline{\Omega}$.\\
\noindent
For $x,y \in \R^d$, $x \times y$ denotes the vector in $\R^d$ defined by
\begin{equation}
x \times y: = (x_1 y_1, \ldots, x_d y_d)^t.
\end{equation}

\part{Variational principles and dissipative systems}
\label{part2}
\setcounter{section}{0}
In this section we discuss the problem of constructing a variational principle for dissipative systems. We review some past issues to solve this problem. Starting from the fact that irreversibility was the main obstruction, we propose to use a doubling phase space which takes into account the evolution of the system toward past or future and different time derivatives operators for each of these variables. We also discuss Riewe's approach to dissipative systems using the fractional calculus and prove that it does not give a satisfying solution. This part can be avoided in a first lecture and must be considered as an heuristic support for the mathematical framework of asymmetric fractional 
calculus of variation developed in the next part.
 \section{Obstruction to variational principles}
The classical {\it inverse problem of the calculus of variations} is the following: Having a given set of differential or partial differential equations is it possible to know if they 
derive from 
a variational principle, {\it i.e.} 
as a critical point of an explicit Lagrangian functional. 
There exists a huge literature on this subject.
The Helmholtz's conditions give necessary and sufficient conditions for a given set of equations to be obtained as an Euler-Lagrange equation. We refer to \cite{Olver} for more details. Helmholtz's conditions are algebro-analytic and are related to the self adjointness of the differential operator attached to the equations. However, they do not give an idea of the physical origin of the obstruction to the existence of a variational principle.

For dissipative systems a classical result of Bauer, \cite{Bauer}, 
in 1931 states that a linear set of differential equations with 
constant coefficients cannot 
derive from a variational principle. The main obstruction is precisely the dissipation of energy which induces a non reversible dynamic in time. Bateman, 
\cite{Ba}, pointed out that this obstruction is only valid if one understand that the variational principle does not produce additional equations. In particular, Bateman constructs a complementary set of equations  which enables him to find a variational formulation. The main idea behind Bateman's approach is that a dissipative system must be seen as physically incomplete. He does not give any additional comments,
but we think that it is related to the irreversibility of the underlying dynamics. An extension of Bateman's construction has been recently given for nonlinear evolution equations \cite{gtd}.

Contrary to dissipative systems,
for  a classical reversible system, the time evolution in the two directions
(past and future) of time is well described by a single 
differential equation.
In fact, we do not know what is the dynamical evolution if we reverse the arrow of time for a dissipative system. This is for example 
the case of the diffusion equation or the convection-diffusion equation. 
Actually, 
the diffusion phenomenon is non reversible in time.
\section{Dealing with irreversibility: doubling phase space and operators}
Bauer's theorem and Bateman construction give two distinct ways to overcome the obstruction to a variational formulation: 
\begin{itemize}
\item To avoid Bauer's obstruction to the existence of a variational principle, a natural idea is to deal with a different kind of differential calculus. This is precisely what we will do in the next part, following a previous attempt of  
Riewe (\cite{Riewe1},\cite{Riewe2}) where the Riemann-Liouville and Caputo fractional calculus are used. However, as we will see, the attempt of Riewe does not work and an appropriate modification of his formalism has to be done.

\item The lost of reversibility is interpret by Bateman as an incomplete description of a system. We suggest to say that we have two different dynamical variables $(x_+, x_- )$ which represent the evolution of the systems toward future or past. In experiments we have only access to the time evolution $x_+$, i.e. that a physical process can be modelled by a curve 
$t\mapsto x_+ (t)$. The evolution of the same physical process if we reverse the arrow of time is not known.
But this doubling of phase space is only one part of the dynamical modelization of a physical process. In experiment, we need an operator acting on the phase variables which controls the dynamical behavior. In classical mechanics, this is done using the ordinary differential calculus and in particular the classical derivative with respect to time $d/dt$. In the new setting, one must look for two time operators $d_+$ and $d_-$ acting on $x_+$ and $x_-$ which we call left and right time operators in the sequel. 
\end{itemize}

These simple remarks tell us that if we look for a variational formulation of a general set of equations then the functional must at least be defined on the doubled phase space $X=(x_+ ,x_- )$ with a differential operator $D$ 
given by $DX =(\Dpv{\alpha} x_+ ,\Dmv{\alpha} x_- )$ where $\Dpv{\alpha}$ and 
$\Dmv{\alpha}$ are right and left Riemann-Liouville or Caputo derivatives.

\section{Riewe's approach to dissipative systems: fractional mechanics}
In 1996-1997, Riewe (\cite{Riewe1},\cite{Riewe2}) defined a {\it fractional Lagrangian} framework to deal with dissipative systems. 
Riewe's theory follows from a simple observation: 
''{\it If the Lagrangian contains terms proportional 
to $\di\left ( {d^n x\over dt^n} \right )^2$, 
then the Euler-Lagrange equation will have a term proportional 
to $\di {d^{2n} x \over dt^{2n}}$. Hence a frictional 
force $\di\gamma {dx\over dt}$ should derive
from a Lagrangian containing a term proportional to the 
fractional derivative  $\di\left ( {d^{1/2}x \over dt^{1/2}} \right )^2$}'' where the notation $\di \frac{d^{1/2}}{ dt^{1/2}}$ 
represents formally an operator satisfying the composition 
rule $\di {d^{1/2} \over dt^{1/2}} \circ {d^{1/2} \over dt^{1/2}} 
={d\over dt}$.  
He then studied fractional Lagrangian functional using the left and right Riemann-Liouville fractional derivatives which satisfy the composition rule formula. 
Let $L$ a Lagrangian defined by:
\begin{eqnarray*} 
L:\, \R \times \R \times \R \times \R & \longrightarrow & \R ,\\ 
(x,v_+ ,v_- ,w ) & \longmapsto & L(x,v_+ ,v_- ,w ) 
\end{eqnarray*}
and $x: [a,b] \rightarrow \R$ a trajectory.
With this notation the functional studied by Riewe is given by
$$\mathcal{L} (x)=
\int_a^b L(x,\Dpv{1/2} x , \Dmv{1/2} x , \frac{dx}{dt} ) dt .$$
He proved that critical points $x$ of this functional correspond to the solutions of the generalised fractional 
Euler-Lagrange equation 
$$\di {d\over dt} \left ( 
\di {\partial L \over \partial w} (\star^{1/2} ) \right ) 
+ \Dmv{1/2} \left ( 
\di {\partial L \over \partial v_+} (\star^{1/2} ) \right )
* \Dpv{1/2} \left ( 
\di {\partial L \over \partial v_-} (\star^{1/2} ) \right ) =\di {\partial L\over \partial x} (\star^{1/2} ),$$
where $\star^{1/2} := (x,\Dpv{1/2} x , \Dmv{1/2} x , \frac{dx}{dt} )$.

Riewe derived such a generalised Euler-Lagrange equation for more general functionals depending on left and right Riemann-Liouville derivatives $\Dmv{\alpha}$ and $\Dpv{\alpha}$ with arbitrary $\alpha >0$ (see \cite{Riewe1}, equation (45) p. 1894). 

The main drawback of this formulation
is that the dependence of $L$ with respect to $\Dpv{1/2}$ 
(resp. $\Dmv{1/2}$) induces a derivation with respect to $\Dmv{1/2}$ 
(resp. $\Dpv{1/2}$) in the equation. 
As a consequence, we will always obtain mixed terms of the form $\Dmv{1/2} \circ \Dpv{1/2} x$ or $\Dpv{1/2} \circ \Dmv{1/2} x$
in the associated Euler-Lagrange equation. 
For example, if we consider the Lagrangian 
$$L(x,v_+ ,v_- ,w)=\di {1\over 2} m w^2 -U(x) +\di {1\over 2} \gamma v_+^2 ,$$
we obtain as a generalised Euler-Lagrange equation
$$\di m \di {d^2 x\over dt^2} +\gamma \Dmv{1/2} \circ \Dpv{1/2} x +U' (x)=0.$$
However, in general 
$$\Dmv{1/2} \circ \Dpv{1/2} x \not= \di {dx\over dt},$$
so that this theory cannot be used in order to provide a variational principle for the linear friction problem. This problem of the mixing between the left and right derivatives in the fractional calculus of variations is well known (see for example Agrawal \cite{Agrawal}). It is due to the {\it integration by parts} formula which is given for $f$ and $g$ in $C_0^0 ([a,b])$ by
$$\int_a^b f(t) \, \Dmv{\alpha} g(t) dt 
=\int_a^b \Dpv{\alpha} f(t) \, g(t) dt .$$
In (\cite{Riewe1}, p.1897) Riewe considered the limit $a\rightarrow b$ while keeping $a<b$. He then approximated $\Dmv{\alpha}$ by $\Dpv{\alpha}$. However, this approximation is not justified in general for a large class of functions so that Riewe's derivation of a variational principle for the linear friction problem is not valid. 

In \cite{cr2}, Cresson  
tried to overcome this problem by modifying the underlying set of variations in the fractional calculus. The set of variations is made of functions $h$ satisfying $\Dmv{\alpha} h =\Dpv{\alpha} h$. A critical point of the fractional functional under this restriction is called a weak critical point. In that case, we obtain that solutions of the linear friction problem corresponds to weak critical point of the functional but not an equivalence as in the usual case. The main problem is that the set of variations is too small to derive a Dubois-Raymond result. We refer in particular to the work of Klimek, \cite{Klimek},
 for more details. 

In the next section, we review the main result of \cite{ci} allowing us to obtain an equivalence between solutions of the linear friction problem and critical points of a fractional functional. 

\part{Asymmetric fractional calculus of variations}
\label{part3}
\setcounter{section}{0}
In this part, we recall the classical definitions of the left and 
right Riemann-Liouville and Caputo derivatives in the one 
dimensional case. We define the multidimensional fractional 
analogous for partial derivatives, gradient and divergence. In particular, we prove a fractional Green-Riemann theorem. We give a simplified version of the asymmetric fractional calculus of variations introduced in \cite{ci} for which we refer for more details and weaker assumptions on functional spaces. This formalism will be used in the last part to derive the convection-diffusion equation from a variational principle.
\section{Fractional calculus}
\label{sect:}
\subsection{Fractional operators: the one-dimensional case}
For a general overview of the fractional calculus and more details 
we refer to the classical book of Samko, Kilbas and Marichev, \cite{Samko}.
\subsubsection{Fractional integrals}
Let $\beta >0$, and $f$, $g$ : $[a,b] \rightarrow \R^n$.
\begin{defi}
The left and right Riemann-Liouville fractional integrals of $f$ are respectively defined by
$$
\left .
\begin{array}{lll}
(\Ipv{\beta} f)(t) & = & \di {1\over \Gamma (\beta )} \di \int_a^t (t-\tau )^{\beta -1} f(\tau ) d\tau , \\
(\Imv{\beta} f)(t)& = & \di {1\over \Gamma (\beta )} \di \int_t^b (\tau -t)^{\beta -1} f(\tau ) d\tau , \\
\end{array}
\right .
$$
for $t\in [a,b]$, where $\Gamma$ is the Gamma function.
\end{defi}
\subsubsection{Fractional derivatives}

Let $\alpha >0$. Let $p\in\N$ such that $p-1 \leq \alpha <p$.
\begin{defi}
Let $t\in [a,b]$, the function $f$ is left (resp. right) 
Riemann-Liouville  differentiable of order $\alpha$ at $t$ 
if the quantity
$$\Dpv{\alpha} f(t)= \left ( \di {d^p \over dt^p} \circ \Ipv{p-\alpha} \right ) f(t)\ \  
\; \text{and}\quad
\Dmv{\alpha} f(t)= \left ( (-1)^p \di {d^p \over dt^p} \circ \Imv{p-\alpha} \right ) f(t),$$
exist respectively. The value $\Dpv{\alpha} f(t)$ (resp. $\Dmv{\alpha} f(t)$) is called the left (resp. right) Riemann-Liouville fractional derivative of order $\alpha$ of $f$ at $t$.
\end{defi}  
Exchanging the order of composition we obtain the left and 
right Caputo fractional derivatives:
\begin{defi}
Let $t\in [a,b]$, the function $f$ is left (resp. right) Caputo 
differentiable of order $\alpha$  at $t$ if the quantity
$${}^c\Dpv{\alpha} f(t)= \left ( \di  \Ipv{p-\alpha} \circ {d^p \over dt^p} \right ) f(t)\ \
\; \text{ and}\quad
{}^c\Dmv{\alpha} f(t)= \left (  \di \Imv{p-\alpha} \circ (-1)^p {d^p \over dt^p} \right ) f(t),$$
exist respectively. The value ${}^c\Dpv{\alpha} f(t)$ (resp. ${}^c\Dmv{\alpha} f(t)$) is called the left (resp. right) fractional derivative of order $\alpha$ of $f$ at $t$.
\end{defi}  

In a general setting the subscript ${}^c$ refers to Caputo derivative, whereas no subscript is used to specify the Riemann-Liouville 
fractional derivative. 
\subsection{Properties}
Let us mention the following relations linking both Riemann-Liouville and Caputo fractional derivatives: 
\begin{lem} 
\label{caputo-RL}
Let $0\leq \alpha <1$ and $f \in AC^1 ([a,b])$ then $\Dpv{\alpha}$ and $\Dmv{\alpha}$ exist almost everywhere and 
\begin{eqnarray*}
\Dpv{\alpha} f(t) = {}^c\Dpv{\alpha} f(t) +\frac{(t-a)^{-\alpha}}{\Gamma(1-\alpha)}f(a)\\
\Dmv{\alpha} f(t) = {}^c\Dmv{\alpha} f(t) +\frac{(b-t)^{-\alpha}}{\Gamma(1-\alpha)}f(b)\,.
\end{eqnarray*}
\end{lem}
The proof can be found in (\cite{Samko}, thm.2.2, p.39).
The following lemma concerns the semi-group property and the integration by parts formula for the fractional derivatives of both Riemann-Liouville and Caputo types:
\begin{lem} 
\label{group-part}
\begin{enumerate}
\item Composition rule formula: Let $f\in AC^2 ([a,b])$, then
\begin{eqnarray*}
 {}^c\Dpv{1/2} \circ  {}^c\Dpv{1/2}  =  \frac{d}{dt} & \quad
& \Dpv{1/2}\circ  {}^c\Dpv{1/2} =  \frac{d}{dt}.
\end{eqnarray*}
\item Integration by parts formula: 
Let $0\leq \alpha <1$. 
Let $f\in AC^1 ([a,b])$ and $g\in C_0^1 ([a,b])$ then
\begin{eqnarray*}
\int_a^b ({}^c\Dpv{\alpha} f)(t)g(t)dt & =  
& \int_a^b f(t)({}^c\Dmv{\alpha} g)(t)dt \\
\int_a^b (\Dpv{\alpha} f)(t)g(t)dt & =  
& \int_a^b f(t)({}^c\Dmv{\alpha} g)(t)dt .
\end{eqnarray*}
\end{enumerate}
\end{lem}
We refer to \cite{Samko}, p.46 for a proof.
\begin{lem}[Regularity]
\label{regularity}
Let $\alpha >0$ and $p\in \N$.
\begin{enumerate}
\item If $f\in C^1 ([a,b])$, then ${}^c \Dpv{\alpha} f \in C^0_+ ([a,b])$.
\item If $f\in C^1_+ ([a,b])$, then $\Dpv{\alpha} f ={}^c\Dpv{\alpha} f$, and $\Dpv{\alpha} f \in C^0_+ ([a,b])$.
\item If $f\in C^{p+1}_+ ([a,b])$, 
then ${}^c \Dpv{\alpha} f \in C^p_+ ([a,b])$.
\end{enumerate}
\end{lem}
The previous lemma is stated under more general but implicit assumptions 
(as $f\in \Ipv{\alpha}  (L^1 )$) in \cite{Samko}. However, as we need explicit conditions on the functional spaces in order to develop the fractional calculus of variations, we give a less general result but with explicit functional spaces.
\begin{proof}
\begin{enumerate}
\item As $\di\frac{df}{dt} \in C^0 ([a,b])$, 
we have $\di\Ipv{\alpha}\bigl( \frac{df}{dt}\bigl) \in C^0_+ ([a,b])$ 
using (Thm. 3.1 of \cite{Samko}, p.53 with $\lambda=0$).  

\item It results from 1. and  lemma \ref{caputo-RL}.

\item Let $1\leq k \leq p$. As $f\in C^{p+1}_+ ([a,b])$, then 
$\di\frac{df}{dt}\in C^k_+ ([a,b])$ 
$$\di {d^k \over dt^k} {}^c \Dpv{\alpha}  f = \di {d^k \over dt^k} \Ipv{1-\alpha} \frac{df}{dt} =\Ipv{1-\alpha} f^{(k+1)} .$$
As $f^{(k+1)} \in C^0 ([a,b])$, we have $\Ipv{1-\alpha} f^{(k+1)} \in C^0_+ ([a,b])$ using (Thm. 3.1 of \cite{Samko}, p.53 with $\lambda=0$). Then ${}^c \Dpv{\alpha} f \in C^k ([a,b])$ and $\di {d^k \over dt^k} {}^c \Dpv{\alpha} f (a) =0$ for $1\leq k\leq p$. Moreover ${}^c \Dpv{\alpha} f (a)=0$ by 1., so that ${}^c \Dpv{\alpha} f \in C^p_+ ([a,b])$.
\end{enumerate}
\end{proof}
A similar result holds for the right derivative.
\subsection{Fractional operators: the multidimensional case}
We generalise the previous definitions and properties on $\Omega$ by introducing the multidimensional fractional operators.

\subsubsection{Definitions and notations}

For a function $v:\oo \rightarrow \R$, we denote by $v_{i,x}$ the function  
defined on $\Omega_{i,x}$, as in \eqref{omegaix} acting in the $i$-th component of $v$.
The $i$-th partial fractional derivatives are given by:
\begin{equation*}
\partial^\alpha_i v(x) := \Dpv{\alpha} v_{i,x}(x_i), \quad {}^c \partial^\alpha_i v(x): = {}^c\Dpv{\alpha} v_{i,x}(x_i),
\end{equation*} 
the right fractional Riemann-Liouville and Caputo partial 
derivatives with respect to the $i$-th component of $v$.
In a same way the left fractional Riemann-Liouville and Caputo partial derivatives are given by
\begin{equation*}
\overline{\partial}^\alpha_i v(x) := \Dmv{\alpha} v_{i,x}(x_i), \quad \overline{{}^c \partial}^\alpha_i v(x) := {}^c\Dmv{\alpha} v_{i,x}(x_i).
\end{equation*} 
The associated Riemann-Liouville and Caputo {\it fractional gradient} of $u$, denoted by $\nabla^\alpha u$ and ${}^c \nabla^\alpha u$, are defined by
\begin{equation*}
\begin{array}{cc}
\nabla^\alpha u(t,x) := 
\begin{pmatrix}
\partial^\alpha_1 u(t,x) \\
\vdots \\
\partial^\alpha_d u(t,x)
\end{pmatrix},
&\quad
{}^c \nabla^\alpha u(t,x) := 
\begin{pmatrix}
{}^c \partial^\alpha_1 u(t,x) \\
\vdots \\
{}^c \partial^\alpha_d u(t,x)
\end{pmatrix}.
\end{array}
\end{equation*}  
The Riemann-Liouville and Caputo {\it fractional divergence} 
of $v \, : \, \Omega \rightarrow \R^d$ are defined by 
\begin{equation*}
\divg^\alpha v(x) = \sum_{i=1}^d \partial^\alpha_i v_i(x)\quad
{}^c \divg^\alpha v(x) = \sum_{i=1}^d {}^c \partial^\alpha_i v_i(x).
\end{equation*} 
The analogous definitions of the left fractional gradient and 
divergence hold adding a bar on the previous symbols.
\begin{lem}
\label{lem:div-grad}
Let $\gamma=(\gamma_1 ,\dots ,\gamma_d) \in \R^d$, $u \in AC^2 (\overline{\Omega})$ and $x \in \Omega$, then we have
\begin{equation*}
\label{div-grad}
\divg^{1/2} \left( \gamma \times {}^c \nabla^{1/2} u(x) \right) = \gamma \cdot \nabla u(x).
\end{equation*} 
\end{lem}
\begin{proof}
Since $\gamma \times {}^c \nabla^{1/2} u = 
\begin{pmatrix}
\gamma_1 \, {}^c \partial^{1/2}_1 u  \\
\vdots  \\
\gamma_d \, {}^c \partial^{1/2}_d u
\end{pmatrix} ,
$
then 
\begin{equation*}
\divg^{1/2} \left( \gamma \times {}^c \nabla^{1/2} u(x) \right) = 
\sum_{i=1}^d \gamma_i \, \left( \partial^{1/2}_i \circ {}^c \partial^{1/2}_i \right) u(x).
\end{equation*}
Applying lemma \ref{group-part} concludes the proof. 
\end{proof}
\subsubsection{Fractional Green-Riemann formula}
The following lemma is a consequence of the unidimensional 
integration by parts formula of lemma \ref{caputo-RL} and leads 
to the fractional version of the Green-Riemann 
theorem mixing Caputo and Riemann-Liouville derivatives given in lemma 
\ref{green_frac}.
\begin{lem}
\label{lemmeintermediaire}
Let 
$u\in AC^1 (\overline{\Omega} )$ and $v\in C^1_0 (\Omega )$ then
$$
\int_{\Omega} u(x)  \overline{{}^c \partial}^\alpha_i v(x) dx =\int_{\Omega}  \partial^\alpha_i u (x)  v(x) dx .
$$
\end{lem}

\begin{proof}
We first extend $u$, $v$, $\overline{{}^c \partial}^\alpha_i v$ and 
$\partial^\alpha_i u$ over $\R^d$ by associating the value $0$ 
if $x\in \R^d \setminus \overline{\Omega}$. In this case, we have
$$
\left .
\begin{array}{lll}
\di\int_{\Omega} u(x) \, \overline{{}^c \partial}^\alpha_i v(x) dx & = & \di\int_{\R^d} u(x) \, \overline{{}^c \partial}^\alpha_i v(x) dx \\
 & = & \di\int_{\R^{d-1}} \left [ \int_{x_i =-\infty}^{\infty} u(x) \, \overline{{}^c \partial}^\alpha_i v(x) dx_i \right ] dx_1 \dots dx_{i-1} dx_{i+1} \dots dx_d .
\end{array}
\right .
$$
As $\Omega$ is convex, there exists $a_{i,x}$ and $b_{i,x}$ such that 
$a_{i,x} \leq b_{i,x}$ and 
\begin{equation*}
\int_{x_i =-\infty}^{\infty} u(x) \, \overline{{}^c \partial}^\alpha_i v(x) dx_i 
=\int_{a_{i,x}}^{b_{i,x}} u_{i,x} (x_i ) {}^c\Dmv{\alpha} v_{i,x}(x_i).
\end{equation*}
As $u_{i,x} \in AC^1 ([a_{i,x},b_{i,x}])$ and 
$v_{i,x} \in C_0^1 ([a_{i,x},b_{i,x}])$, 
the integration by parts formula from lemma \ref{group-part} leads to:
\begin{eqnarray*}
\di\int_{a_{i,x}}^{b_{i,x}} u_{i,x} (x_i ) {}^c\Dmv{\alpha} v_{i,x} (x_i ) dx_i & = & \di\int_{a_{i,x}}^{b_{i,x}} {}^c\Dpv{\alpha} u_{i,x} (x_i ) v_{i,x} (x_i ) dx_i \\
 &= & \di\int_{a_{i,x}}^{b_{i,x}} \partial_i^{\alpha} u (x) v_{i,x} (x_i ) dx_i \\
 & = &  \di\int_{x_i=-\infty}^{+\infty} \partial_i^{\alpha} u (x) v(x) dx_i .
\end{eqnarray*}
As a consequence, we have 
\begin{eqnarray*}
\di\int_{\R^d} u(x) \, \overline{{}^c \partial}^\alpha_i v(x) dx_i & = & \di\int_{\R^{d-1}} \left [ \int_{x_i =-\infty}^{\infty} \partial_i^{\alpha} u(x)  v(x) dx_i \right ] dx_1 \dots dx_{i-1} dx_{i+1} \dots dx_d ,\\
 & = & \di\int_{\Omega} \partial_i^{\alpha} u(x) \, v(x) dx. 
\end{eqnarray*}
\end{proof}

\begin{lem}[Fractional Green-Riemann theorem]
\label{green_frac}
Let $\Omega \in \R^d$ be a smooth $d$-dimen\-sional bounded convex domain, $u \in C^1_0 (\overline{\Omega},\R)$ and $v=(v_1 ,\dots ,v_d ) \in AC^1 (\Omega)^d$, then we have
\begin{equation}
\int_\Omega v(x) \cdot \overline{{}^c \nabla}^\alpha u(x) \, dx = \int_\Omega \divg^\alpha (v(x)) \, u(x) \, dx .
\end{equation} 
\end{lem}
\begin{proof}
The  proof results from lemma \ref{lemmeintermediaire}.
On the canonical basis the scalar product is given by
\begin{equation*}
\int_\Omega v(x) \cdot \overline{{}^c \nabla}^\alpha u(x) \, dx = \sum_{i=1}^d \int_\Omega v_i(x) \, \overline{{}^c \partial}^\alpha_i u(x) \, dx.
\end{equation*} 
Applying lemma \ref{lemmeintermediaire}, we have for any $1 \leq i \leq d$,
\begin{equation*}
\int_\Omega v_i(x) \, \overline{{}^c \partial}^\alpha_i u(x) \, dx = \int_\Omega  \partial^\alpha_i v_i(x) \, u(x) \, dx .  
\end{equation*} 
The definition of $\div^{\alpha}$ concludes the proof.
\end{proof}

\section{Asymmetric fractional calculus of variations}
\label{sect:asymetric}
The fractional Euler-Lagrange equations obtained so far in 
\cite{Riewe1, Agrawal, cr2, Frederico_Torres} 
involve both left and right fractional derivatives. This is a main drawback, if ones want to recover PDEs with order one derivative as composition of fractional derivatives of order 1/2. In this section, we give a simplified version of the asymmetric calculus of variations introduced in \cite{ci} which 
provides causal fractional Euler-Lagrange equations. We refer to \cite{ci} and \cite{inizan} for more details and in particular for weaker 
assumptions on the functional spaces.
\subsection{Asymmetric fractional Lagrangian}
Here the fractional derivatives are seen as partial fractional 
derivatives according to $t$ and $x$. 
With the notations $u_x$ and $u_t$ from section 
\ref{part1} \ref{subsection:functional} 
we denote by
\begin{equation*}
\Dcp u(t,x) := \Dcp u_x(t), \quad  
{}^c\nabla^{\alpha} u(t,x) := {}^c\nabla^{\alpha} u_t(x) \quad
\text{ and} \quad\nabla u(t,x) := \nabla u_t(x).
\end{equation*}
For a field $U=(u_+,u_-)$ we define
\begin{equation*}
\Da U(t,x) := \bigl (\Dcp u_+(t,x), - \Dcm u_-(t,x)\bigl )
\quad  
{}^c\nabla^{\alpha} U(t,x) := \bigl({}^c\nabla^{\alpha} u_+(t,x), - 
{}^c\bar{\nabla^{\alpha}} u_-(t,x)\bigl).
\end{equation*} 
For a generalised Lagrangian $L(t,x,y,v,w,z)$, 
we denote $\partial_y L$, $\partial_v L$, $\partial_w$ and $\partial_z$ 
the partial derivatives of $L$.
\begin{defi}
\label{defi:L}
Let be a Lagrangian $L$ defined as follows:
\begin{equation*}
\begin{array}{cccl}
L \, : & [a,b]\times \oo \times  \R \times \R \times \R^d \times \R^d & \longrightarrow & \; \R  \\
               &   (t,x,y,v,w,z)    & \longmapsto & L(t,x,y,v,w,z)
\end{array}
\end{equation*} 
The asymmetric representation of $L$ is denoted by $\tilde L$ and given by:
\begin{equation*}
\begin{array}{cccl}
\tilde L \, : & [a,b]\times \oo \times  \R^2 \times \R^2 \times \R^{2d} \times \R^{2d} & \longrightarrow & \; \R  \\
               &   (t,x,U,V,W,Z)    & \longmapsto & \tilde L(t,x,U,V,W,Z)
\end{array}
\end{equation*}
where 
\begin{equation*}
 \tilde L\bigl(t,x,(u_+,u_-),(v_+,v_-),(w_+,w_-),(z_+,z_-)\bigl)
:=L(t,x,u_++u_-,v_++v_-,w_++w_-,z_++z_- ).
\end{equation*}
\end{defi}
\noindent
The asymmetric fractional Lagrangian functional $\cal{L_{\alpha}}$ 
is defined by: 
\begin{defi}
Let $L$ be a Lagrangian as defined in definition \ref{defi:L} of 
class $C^1 ([a,b] \times \R^{3d+2} )$  
and $\tilde{L}$ its asymmetric representation. 
The associated asymmetric fractional Lagrangian 
functional of order $\alpha$ is given by
\begin{equation*} 
\begin{array}{cccl}
\mc{L}_\alpha \, : & C^1([a,b]\times \Omega  )^2 & 
\longrightarrow & \quad \R  \\
&   U   & \longmapsto & \displaystyle{\int_a^b \int_{\oo}} 
\tilde{L} \left( t,x, U(t,x), \Da U(t), {}^c\nabla^{\alpha} U(t,x),\nabla U(t,x) \right) \, dx\,dt.
\end{array}
\end{equation*}
\end{defi}
For convenience, we denote by $L_{\alpha}$ the 
functional defined for any
$U=(u_+,u_-) \in C^1([a,b]\times\Omega  )^2$ by
\begin{equation}
\label{eq:Lalpha}
\begin{array}{ccc}
L_\alpha(U)(t,x) & := &\tilde{L} \left( t,x, U(t,x), \Da U(t), {}^c\nabla^{\alpha} U(t,x),\nabla U(t,x) \right)\\
 & = & 
L(t,x, u_+(t,x) + u_-(t,x), \Dcp u_+(t,x) - \Dcm u_-(t,x), \\
& & {}^c \nabla^{\alpha} u_+(t,x) - {}^c \overline{\nabla}^{\alpha} u_-(t,x), \nabla u_+(t,x)+\nabla u_-(t,x)),
\end{array}
\end{equation} 
for any $x\in \Omega$ and $t\in [a,b]$ so that 
$$
\mc{L}_\alpha (U) =\di\int_a^b \di\int_{\Omega} L_{\alpha} (U) (t,x) dx \, dt .
$$
Let us note that as 
$U\in C^1 ([a,b]\times\Omega )$ we have 
$\Da_ U(t) $, ${}^c\nabla^{\alpha}_x U(t,x)$ and 
$\nabla_x U(t,x) \in C^0 ([a,b] \times \Omega )$. 
Using the fact that $L$ is $C^1 ([a,b]\times \R^{3d+2})$ we 
obtain that $L_{\alpha} (C^1 ( [a,b]\times\Omega ) 
\subset C^0 ([a,b]\times\Omega)$ and conclude that 
$\mc{L}_{\alpha}$ is well defined. 
\subsection{Asymmetric calculus of variations}
The next lemma explicits the differential of the functional $\mc{L}_\alpha$ defined on $U \in (C^1 ([a,b]\times \Omega ))^2$ in the direction $( C_0^1([a,b]\times\Omega ) )^2$

\begin{lem} 
\label{ext_frac}
Let $U \in (C^1 ([a,b]\times\Omega ))^2$. 
Let $\star^{\alpha} :=  \left(t,x, U(t,x), \Da U(t,x), {}^c\nabla^{\alpha} U(t,x),\nabla U(t,x) \right) .$
We assume that 
\begin{itemize}
\item $\forall \,x \in \Omega, \; t \mapsto \partial_{v} L (\star^{\alpha} ) \in AC^1([a,b])$,
\item $\forall \, t \in [a,b], \; x \mapsto  \partial_w L (\star^{\alpha} ) \in AC^1 (\bar\Omega )$,
\item $\forall \, t\in [a,b],\; x \mapsto \partial_z L (\star^{\alpha} ) \in  C^1(\Omega)$.
\end{itemize}
Then $\mc{L}_\alpha$ is $(C^1 ([a,b]\times \Omega ))^2$-differentiable 
at $U$ and in any direction 
$H=(h_+,h_-) \in (C_0^1 ([a,b] \times\Omega ))^2$,
the differential of $\mc{L}_\alpha$ is given by
$$
\left .
\begin{array}{lll}
\label{DL_alpha}
D\mc{L}_\alpha(U,H) & = & \di\int_a^b \int_{\Omega}
\left[ \partial_y L (\star^{\alpha} ) + \Dm \partial_v L (\star^{\alpha} )  
+ \overline{\mbox{\rm div}}^{\alpha} ( \partial_w {L} (\star^{\alpha} ))
-\div (\partial_z {L} (\star^{\alpha} ))
\right] \cdot h_+(t) \,dx \, dt  \\
 &  &   + \di\int_a^b  \int_{\Omega}
\left[ 
\partial_y L (\star^{\alpha} ) - \Dp \partial_v {L} (\star^{\alpha} )  
- \mbox{\rm div}^{\alpha} ( \partial_w {L} (\star^{\alpha}) )
-\div (\partial_z {L} (\star^{\alpha} ))
\right] \cdot h_-(t) \,dx \, dt .
\end{array}
\right .
$$
\end{lem}

\begin{proof}
Using a Taylor expansion of the Lagrangian $L$, we obtain:
\begin{equation*}
\begin{array}{lll}
D\mc{L}_\alpha(U,H) & = & \di\int_a^b \di\int_{\Omega}
\left[
\partial_y L(\star) \cdot (h_+ + h_-) 
+\partial_v L(\star) \cdot (\Dcp h_+ -\Dcm h_- ) 
\right . \\
 & &  
\left . 
+\partial_w L(\star) \cdot \left ( {}^c\nabla^{\alpha} h_+ - {}^c\overline{\nabla^{\alpha}} h_- \right ) 
+\partial_z L(\star) \cdot (\nabla h_+ +\nabla h_+ )  
\right] 
\,dx \, dt.
\end{array}
\end{equation*}
As $\partial_v L(\star) \in AC^1 ([a,b])$ by assumption, 
using the integration by parts formula of lemma \ref{group-part}
with $h_+$ (resp. $h_-$) in $C_0^1 ([a,b])$ leads to
$$
\di\int_a^b \di\int_{\Omega}
\partial_v L(\star) \cdot (\Dcp h_+ -\Dcm h_- ) 
\, dx\, dt
=
\di\int_a^b \di\int_{\Omega}
\left ( \Dm \partial_v L(\star ) \cdot h_+ 
- \Dp \partial_v L (\star ) \cdot h_- \right )  \, dx\, dt .
$$
As $\partial_w L (\star ) \in AC^1 (\overline{\Omega} )$ 
and $h_+$ (resp. $h_-$) is in $C_0^1 ([a,b] \times \Omega )$, we can apply 
the fractional Green-Riemann formula from lemma \ref{green_frac} to obtain 
$$
\di\int_a^b \di\int_{\Omega}
\partial_w L(\star) \cdot \left ( {}^c\nabla^{\alpha} h_+ - {}^c\overline{\nabla^{\alpha}} h_- \right )
 dx\, dt
=
\di\int_a^b \di\int_{\Omega}
\left ( 
\overline{\mbox{\rm div}}^{\alpha} (\partial_w L(\star ) ) \cdot h_+ - \mbox{\rm div}^{\alpha} (\partial_w L (\star )) \cdot h_- 
\right )
 dx \, dt .
$$
The last part of the formula comes from the usual Green-Riemann theorem. This completes the proof.
\end{proof}
A consequence of the previous lemma is the following characterisation of the extremals of $\mc{L}_\alpha$ for asymmetric fractional functional as solutions of two fractional Euler-Lagrange equations:
\begin{thm}
\label{maintheorem1}
Let $U \in ( C^1( [a,b]\times\Omega ) )^2$. \\
Let $\star^{\alpha} :=  \left(t,x,U(t,x), \Da U(t,x), {}^c\nabla^{\alpha} U(t,x),\nabla U(t,x) \right)$.
We assume that 
\begin{itemize}
\item $\forall \,x \in \Omega, \; t \mapsto \partial_{v} L (\star^{\alpha} ) \in AC^1([a,b])$,
\item $\forall \, t \in [a,b], \; x \mapsto  \partial_w L (\star^{\alpha} ) \in AC^1 (\bar\Omega )$,
\item $\forall \, t\in [a,b],\; x \mapsto \partial_z L (\star^{\alpha} ) \in  C^1(\Omega)$.
\end{itemize}
Then $(C_0^1 ([a,b] \times \Omega ))^2$ extremals of $\mc{L}_\alpha$ 
correspond to solutions of the following set of 
fractional Euler-Lagrange equations:
$$
\left .
\begin{array}{lll}
\partial_u L (\star^{\alpha} ) 
+ \Dm \bigl( \partial_v L (\star^{\alpha} )  \bigl)
+ \overline{\mbox{\rm div}}^{\alpha} \bigl( \partial_w {L} (\star^{\alpha} )\bigl)
-\div \bigl(\partial_z {L} (\star^{\alpha} )\bigl) & = & 0,  \\
\partial_u L (\star^{\alpha} ) 
- \Dp \bigl( \partial_v {L} (\star^{\alpha} ) \bigl) 
- \mbox{\rm div}^{\alpha} \bigl( \partial_w {L} (\star^{\alpha} )\bigl)
-\div \bigl(\partial_z {L} (\star^{\alpha} )\bigl) & = & 0 .
\end{array}
\right .
$$
\end{thm}

\subsection{Specialisation}
This theorem is not applicable as it is for classical PDEs since
as it provides a system of PDEs. However, by restricting our 
attention to extremals over $C^1 ([a,b]\times\Omega  ) \times \{ 0\}$ 
over variations in $\{ 0\} \times C_0^1 ([a,b]\times \Omega  )$, 
we obtain a more interesting version: 
\begin{thm} 
\label{theorem:EL_coh}
Let $u_+ \in C^1 ([a,b]\times\Omega)$. \\
Let 
$\star^{\alpha}_+  := \bigl( t, x, u_+(t,x), \Dcp u_+(t,x), 
{}^c\nabla u_+(t,x),\nabla u_+(t,x) \bigl).$
We assume that
\begin{itemize}
\item $\forall \,x \in \Omega, \; t \mapsto \partial_v L (\star^{\alpha}_+ ) \in AC^1 ([a,b])$,
\item $\forall \, t \in [a,b], \; x \mapsto \partial_w L (\star^{\alpha}_+ ) \in AC^1 (\bar \Omega )$,
\item $\forall \, t\in [a,b],\; x \mapsto \partial_z L (\star^{\alpha}_+ ) \in  C^1(\Omega)$.
\end{itemize}
Then $(u_+,0)$ is a $\{ 0 \} \times C_0^1 ([a,b]\times\Omega )$-extremal of the action $\mc{L}_\alpha$ if and only if $u_+$ satisfies 
\begin{eqnarray*} \label{eq:EL_coh}
 \partial_y L ( \star^{\alpha}_+ ) 
- \Dp \bigl(\partial_{v} L ( \star^{\alpha}_+ ) \bigl)
- \div^{\alpha} \bigl(\partial_w L ( \star^{\alpha}_+ ) \bigl)
-\div \bigl(\partial_z {L} ( \star^{\alpha}_+ ) \bigl)= 0,
\end{eqnarray*}
for any $x \in \Omega$, $t \in [a,b]$.
\end{thm}
The proof is a consequence of lemma \ref{ext_frac}.
\begin{rema}
In the previous part, we have heuristically justified the introduction of a doubled phase space by saying that irreversibility induces a natural arrow of time. Our idea to focus only on curves in $C^1 ([a,b]\times\Omega ) \times \{ 0\}$ is precisely to say that we are interested in one direction of time (here the future). However, the selection of a transverse set for variations to the underlying phase space is not so clear. It means heuristically that the future depends mostly on the virtual variations in the past. More work are needed in this direction.
\end{rema} 

\part{Application to the convection-diffusion equation}
\label{part4}
Let us consider the reaction-convection-diffusion equation 
defined on $[a,b]\times \Omega$ by (\ref{CDunst}):
\begin{equation} 
\label{eq:conv_diff}
\dpar{}{t} u(t,x) +\gamma\cdot \nabla u(t,x) - 
\divg \left( K \cdot \nabla u(t,x) \right) +\beta u(t,x)  =f(t,x).
\end{equation}
with constant coefficients $\gamma \in \R^d$, $K\in\R^{d\times d}$  
and $\beta \in \R$.
As we already mentioned, this equation does not derive from a 
variational principle in the classical sense. Nevertheless
the result of the previous section allows us to overcome this difficulty
and obtain a variational formulation
of the convection-diffusion equation 
by mean of the asymmetric fractional Lagrangian.
Let us defined the extended Lagrangian $L$ given by
\begin{equation*}
\begin{array}{cc cl}
L : & [a,b]\times \oo \times  \R \times \R \times \R^d \times \R^d & \longrightarrow &  \R  \\
     &   (t,x,y,v,w,z)    & \longmapsto & f(t,x)y - \dfrac{1}{2}\beta y^2
+ \dfrac{1}{2} v^2 + \dfrac{1}{2} (\gamma  \times w)   \cdot w - 
\dfrac{1}{2} (K\cdot z) \cdot z .
\end{array}
\end{equation*} 
The direct application of theorem \ref{theorem:EL_coh}  provides that the solutions of the convection-diffusion equation are 
$\{ 0 \} \times C_0^1 ([a,b]\times\Omega )$-extremals of the following
asymmetric fractional functional $\cal{L}_{1/2}$ defined 
for $U=(u_+,u_-)$ by
\begin{equation*}
\begin{array}{cc}
\cal{L}_{1/2}(U) = & \displaystyle\int_a^b \int_{\oo} L \bigl( t,x,u_+( t,x) + u_-( t,x), 
{}^c\Dpv{1/2} u_+( t,x) - {}^c\Dmv{1/2} u_-( t,x), \\
& \qquad \, {}^c \nabla^{1/2} u_+ ( t,x)- \overline{{}^c \nabla}^{1/2} u_-( t,x), 
\, \nabla u_+ ( t,x)+ \nabla u_- ( t,x)\bigl)\,dx\,dt.
\end{array}
\end{equation*}
Namely, the following result holds:
\begin{thm}
Let $u \in F^2 (\bar{\Omega}\times [a,b])$. 
Then $u$ is a solution of 
the convection-diffusion equation \eqref{eq:conv_diff} 
if and only if $(u,0)$ is a $\{0\} \times C_0^1 ([a,b]\times\Omega)$ 
critical point of $\mathcal{L}_{1/2}$.
\end{thm}
\begin{proof}
Let $x \in \Omega$ and $t \in [a,b]$.  Let $t \in [a,b]$. 
The partial derivatives of $L$ verify:
\begin{itemize}
\item $\partial_y L(t,x,u(t,x),\Dcpv{1/2} x(t), {}^c \nabla^{1/2} u(t,x),  \nabla u(t,x)) = f(t,x) - \beta u(t,x)$,
\item $\partial_v L(t,x,u(t,x),\Dcpv{1/2} x(t), {}^c \nabla^{1/2} u(t,x),  \nabla u(t,x)) =  \Dcpv{1/2} u(t,x)$,
\item $\partial_w L(t,x,u(t,x),\Dcpv{1/2} x(t), {}^c \nabla^{1/2} u(t,x),  \nabla u(t,x)) =  \gamma \times {}^c \nabla^{1/2} u(t,x)$,
\item $\partial_z L(t,x,u(t,x),\Dcpv{1/2} x(t), {}^c \nabla^{1/2} u(t,x),  \nabla u(t,x)) = -K\cdot\nabla u(t,x)$.
\end{itemize}
As $u\in F^2 ([a,b]\times\bar\Omega )$ we have that $u_x \in AC^2 ([a,b] )$ and as a consequence $\Dcpv{1/2} u \in AC^1 ([a,b]) $. We have also $u_t \in C^2 (\bar{\Omega} )$ so that using lemma \ref{regularity}, we deduce ${}^c \nabla^{1/2} u \in C^1 (\Omega)$. Moreover $\nabla u \in C^1 ([a,b])$ so that the conditions of theorem \ref{theorem:EL_coh} are fulfilled.
From lemma \ref{group-part}, as $u_x \in AC^2 ([a,b])$ we have 
$\Dpv{1/2} \circ \Dcpv{1/2} u = \dfrac{d}{dt} u$.
Moreover, as $u_t \in C^2 (\bar\Omega )$, lemma \ref{lem:div-grad} 
applies and we have
\begin{equation*}
\divg^{1/2} \left( \gamma \times {}^c \nabla^{1/2} u(x) \right) = \gamma \cdot \nabla u(x).
\end{equation*} 
This concludes the proof.
\end{proof}
\part{Conclusion and perspectives}
The previous result is only an example of PDEs for which the fractional asymmetric calculus of variation provides a Lagrangian variational formulation when this is not possible using the classical calculus of variations. As we previously said in the introduction, variational formulations of PDEs are important both from the theoretical and practical point of view. In that respect our result is far from being complete. 

We have the following list of open problems and perspectives:

\begin{itemize}
\item One must develop the critical point theory associated to our fractional functionals in order to provide results about existence and regularity of solutions for these PDEs. 

\item Our paper as well as \cite{ci} solve the inverse problem of the fractional calculus of variations for some classical or fractional PDEs (classical or fractional diffusion equation, fractional wave equation, convection-diffusion
$\dots$). However, we have no characterization of PDEs admitting a fractional variational formulation in our setting. In the classical case, the Lie approach to ODEs or PDEs as exposed for example in \cite{Olver} provides a necessary criterion known as Helmholtz's conditions. 
A natural idea is to look for the corresponding theory in our case.

\item There exists suitable numerical algorithms to study classical Lagrangian systems called \textit{variational integrators} which are developed for example in \cite{Lubich}, \cite{mars}. The basic idea of a variational integrator is to preserve this variational structure at the discrete level. A natural extension of our work is then to develop variational integrators adapted to our fractional Lagrangian functionals. A first step in this direction has been done in \cite{cbgi} by introducing the notion of discrete embedding of Lagrangian systems. However, this work does not cover continuous fractional Lagrangian systems and uses only classical discretization of the Riemann-Liouville or Caputo derivative by  Grünwald-Leitnikov expansions. However for classical functionals we have extended this point of view to finite-elements and finite-volumes methods, \cite{cgp}. We will discuss the case of continuous fractional Lagrangian systems in a forthcoming paper.
\end{itemize}
Of course of all these problems are far from being solved for the moment. However, it proves that fractional calculus can be useful in a number of classical problems of Analysis and in particular for PDEs where classical methods do not provide efficient tools.

\bibliographystyle{plain}

\begin{thebibliography}{}

\end{thebibliography}


\begin{thebibliography}{10}
\bibitem{Agrawal}
Om~Prakash Agrawal.
\newblock Formulation of {E}uler-{L}agrange equations for fractional
  variational problems.
\newblock {\em J. Math. Anal. Appl.}, 272(1):368--379, 2002.

\bibitem{Ba}
Harry Bateman.
\newblock On dissipative systems and related variational principles.
\newblock {\em Physical Review}, 38(1):815--819, 1931.

\bibitem{Bauer}
P.~S. Bauer.
\newblock Dissipative dynamical systems.
\newblock {\em Proc. Nat. Acad. Sci.}, 17:311--314, 1931.

\bibitem{cbgi}
Loic Bourdin, Jacky Cresson, Isabelle Greff, Pierre Inizan, 
\newblock Variational integrators for fractional Lagrangian systems in the framework of discrete embeddings, 
\newblock Preprint 2011.

\bibitem{cr2}
Jacky Cresson.
\newblock Fractional embedding of differential operators and {L}agrangian
  systems.
\newblock {\em J. Math. Phys.}, 48(3):033504, 34, 2007.

\bibitem{cgp}
Jacky Cresson, Isabelle Greff, and Charles Pierre.
\newblock Coherent discrete embeddings for Lagrangian and Hamiltonian systems.
\newblock {\em Arxiv-1107.0894}, pages 1--24, 2011.

\bibitem{ci}
Jacky Cresson and Pierre Inizan.
\newblock Variational formulations of differential equations and asymmetric
  fractional embedding.
\newblock {\em  Journal of Mathematical Analysis and Applications}, Volume 385, Issue 2, 15 January 2012, Pages 975-997.

\bibitem{Frederico_Torres}
Frederico, Gast{\~a}o S. F. and Torres, Delfim F. M., A formulation of {N}oether's theorem for fractional problems
              of the calculus of variations, J. Math. Anal. Appl.,
334, {2007},{2}, {834--846}.

\bibitem{gtd}
Sk. Golam~Ali, Benoy Talukdar, and Umapada Das.
\newblock Inverse problem of variational calculus for nonlinear evolution
  equations.
\newblock {\em Acta Phys. Polon. B}, 38(6):1993--2002, 2007.

\bibitem{Lubich}
Ernst Hairer, Christian Lubich, and Gerhard Wanner.
\newblock {\em Geometric numerical integration}, volume~31 of {\em Springer
  Series in Computational Mathematics}.
\newblock Springer-Verlag, Berlin, second edition, 2006.
\newblock Structure-preserving algorithms for ordinary differential equations.

\bibitem{inizan}
Pierre Inizan.
\newblock Dynamique fractionnaire pour le chaos hamiltonien.
\newblock PhD thesis.
\newblock 2010.

\bibitem{Klimek}
Malgorzata Klimek.
\newblock On analogues of exponential functions for antisymmetric fractional
  derivatives.
\newblock {\em Comput. Math. Appl.}, 59(5):1709--1717, 2010.

\bibitem{mars}
Jerrold~E. Marsden and Matthew West.
\newblock Discrete mechanics and variational integrators.
\newblock {\em Acta Numer.}, 10:357--514, 2001.

\bibitem{Olver}
Peter~J. Olver.
\newblock {\em Applications of {L}ie groups to differential equations}, volume
  107 of {\em Graduate Texts in Mathematics}.
\newblock Springer-Verlag, New York, second edition, 1993.

\bibitem{Ortiz}
Michael Ortiz.
\newblock A variational formulation for convection-diffusion problems.
\newblock {\em Internat. J. Engrg. Sci.}, 23(7):717--731, 1985.

\bibitem{Riewe1}
Fred Riewe.
\newblock Nonconservative {L}agrangian and {H}amiltonian mechanics.
\newblock {\em Phys. Rev. E (3)}, 53(2):1890--1899, 1996.

\bibitem{Riewe2}
Fred Riewe.
\newblock Mechanics with fractional derivatives.
\newblock {\em Phys. Rev. E (3)}, 55(3, part B):3581--3592, 1997.

\bibitem{Samko}
Stefan~G. Samko, Anatoly~A. Kilbas, and Oleg~I. Marichev.
\newblock {\em Fractional integrals and derivatives}.
\newblock Gordon and Breach Science Publishers, Yverdon, 1993.
\newblock Theory and applications, Edited and with a foreword by S. M.
  Nikol{\cprime}ski{\u\i}, Translated from the 1987 Russian original, Revised
  by the authors.

\bibitem{Vainberg}
M.~M. Vainberg.
\newblock {\em Variational methods for the study of nonlinear operators}.
\newblock Holden-Day Inc., San Francisco, Calif., 1964.
\newblock With a chapter on Newton's method by L. V. Kantorovich and G. P.
  Akilov. Translated and supplemented by Amiel Feinstein.

\end{thebibliography}
\def\cprime{$'$}

\end{document}